\documentclass[reqno]{amsart}%
\usepackage{amstext}
\usepackage{amsfonts}
\usepackage{amsmath}
\usepackage{amssymb}
\usepackage{hyperref}
\usepackage{graphicx}%
\providecommand{\U}[1]{\protect\rule{.1in}{.1in}}
\numberwithin{equation}{section}
\newtheorem{theorem}{Theorem}[section]

\newtheorem{corollary}[theorem]{Corollary}

\begin{document}
\title[Regularity criteria for the 3D Navier-Stokes equations ]{Regularity criteria for strong solutions to the 3D Navier-Stokes equations}

\begin{abstract}
In this paper, we study the regularity problem of the 3D incompressible
Navier--Stokes equations. We prove that the strong solution exists globally
for new regularity criteria. For negligible forces, we give an improvement of
the known time interval of regularity obtained in \cite{9}.

\end{abstract}
\maketitle

\section{Introduction}

Two of the profound open problems in the theory of three dimensional viscous
flows are the unique solvability theorem for all time and the regularity of
solutions. For the three-dimensional Navier-Stokes system weak solutions are
known to exist by a basic result by Leray from 1934 \cite{10}, but the
uniqueness is still open problem \cite{1}-\cite{3} and \cite{8}. Furthermore,
the strong solutions for the 3D Navier-Stokes equations are unique and can be
shown to exist on a certain finite time interval for small initial data and
small forcing term, but the global regularity for the 3D Navier-Stokes is
still open problems (see \cite{4}-\cite{8}, \cite{12}-\cite{14} and references
therein). In 1933 \cite{9}, Leray showed that in the absence of forcing
($f=0$), all solutions of Navier-Stokes equations are eventually smooth (i.e.
after some $T^{\ast}>0$ depending on the data). Kato and Fujita \cite{7}
showed that a smooth solution to the three-dimensional Navier-Stokes equations
exists for all time if $f$ is small in some sense and $u_{0}$ is small in
$H^{\frac{1}{2}}$.

In this paper, we give a new condition for global existence in time for strong
solution for 3D Navier-Stokes equations with external force. We show that no
singularity can occur in finite time for a large class of forcing term. We
also give an extension of the time interval of regularity to the 3D
Navier-Stokes equations with negligible forces \cite{3, 5, 11, 13}. This
result means that the solution does not blow up at $T^{\ast}$.

\section{Notations and preliminaries}

In this section we introduce notations and the definitions of standard
functional spaces that will be used throughout the paper. We denote by
$H_{per}^{m}\left(  \Omega\right)  $, the Sobolev space of periodic functions.
These spaces are endowed with the inner product%
\[
\left(  u,v\right)  =%
{\textstyle\sum\limits_{\left\vert \beta\right\vert \leq m}}
(D^{\beta}u,D^{\beta}v)_{L^{2}\left(  \Omega\right)  }\text{ and the norm
}\left\Vert u\right\Vert _{m}=%
{\textstyle\sum\limits_{\left\vert \beta\right\vert \leq m}}
(\left\Vert D^{\beta}u\right\Vert _{L^{2}\left(  \Omega\right)  }^{2}%
)^{\frac{1}{2}}.
\]
We define the spaces $V_{m}$ as completions of smooth, divergence-free,
periodic, zero-average functions with respect to the $H_{per}^{m}$ norms.
$V_{m}^{\prime}$ denote the dual space of $V_{m}$.

Let $P$ be the orthogonal projection in $L_{per}^{2}\left(
\mathbb{R}
^{3}\right)  ^{3}$ with the range $V_{0}$.\newline Let $A=-P\triangle$ the
Stokes operator. It is easy to check that $Au=-\triangle u$ for every $u\in
D\left(  A\right)  $. We recall that the operator $A$ is a closed positive
self-adjoint unbounded operator.

The eigenvalues of $A$ are $\left\{  \lambda_{j}\right\}  _{j=1}^{j=\infty}$,
$0$ $<$ $\lambda_{1}\leq\lambda_{2}\leq...$and the corresponding orthonormal
set of eigenfunctions $\left\{  w_{j}\right\}  _{j=1}^{j=\infty}$ is complete
in $V_{0}$%
\begin{equation}
Aw_{j}=\lambda_{j}w_{j},\ \ \ w_{j}\in D(A),\forall j. \label{1}%
\end{equation}
Let us now define the trilinear form $b(.,.,.)$ associated with the inertia
terms%
\begin{equation}
b\left(  u,v,w\right)  =\sum_{i,j=1}^{3}%
{\displaystyle\int\limits_{\Omega}}
u_{i}\frac{\partial v_{j}}{\partial x_{_{i}}}w_{j}dx. \label{2}%
\end{equation}
The continuity property of the trilinear form enables us to define (using
Riesz representation Theorem) a bilinear continuous operator $B\left(
u,v\right)  $; $V_{2}\times V_{2}\rightarrow V_{2}^{\prime}$ will be defined
by
\begin{equation}
\left\langle B\left(  u,v\right)  ,w\right\rangle =b\left(  u,v,w\right)
,\text{ }\forall w\in V_{2}\text{.} \label{3}%
\end{equation}
Recall that for $u$ satisfying $\nabla.u=0$ we have%
\begin{equation}
b\left(  u,u,u\right)  =0\text{ and }b\left(  u,v,w\right)  =-b\left(
u,w,v\right)  \text{.} \label{4}%
\end{equation}
We recall some well known inequalities that we will be using in what
follows.\newline Young's inequality%
\begin{equation}
ab\leq\frac{\sigma}{p}a^{p}+\frac{1}{q\sigma^{\frac{q}{p}}}b^{q}%
,a,b,\sigma>0,\text{ }p>1,\text{ }q=\frac{p}{p-1}. \label{5}%
\end{equation}
Poincar\'{e}'s inequality%
\begin{equation}
\lambda_{1}\left\Vert u\right\Vert ^{2}\leq\Vert A^{\frac{1}{2}}u\Vert
^{2}\text{\ for all }u\in V_{0}\text{,} \label{6}%
\end{equation}
$\lambda_{1}$ is the first eigenvalue of the Stokes operator.

\section{Navier-Stokes equations}

The conventional Navier-Stokes system can be written in the evolution form%
\begin{equation}%
\begin{array}
[c]{c}%
\dfrac{\partial u}{\partial t}-\nu\triangle u+u.\nabla u=f,\text{ }t>0,\\
\text{div }u=0,\text{ in }\Omega\times\left(  0,\infty\right)  \text{ and
}u\left(  x,0\right)  =u_{0},\text{ in }\Omega\text{.}%
\end{array}
\label{7}%
\end{equation}
We recall that a Leray weak solution of the Navier-Stokes equations is a
solution which is bounded and weakly continuous in the space of periodic
divergence-free $L^{2}$ functions, whose gradient is square-integrable in
space and time and which satisfies the energy inequality. The proof of the
following theorem is given in \cite{13}.

\begin{theorem}
Assume that $f\in L^{2}(0,T;V_{1}^{\prime})$ and $u_{0}\in V_{0}$ are given.
Then there exists at least one solution $u$ of $(\ref{7})$\ such that $u\in
L^{2}(0,T;V_{1})\cap L^{\infty}\left(  0,T;V_{0}\right)  ,$ $\forall T>0$.
\end{theorem}

For strong solutions, we have the following result \cite{13}.

\begin{theorem}
Assume that $u_{0}\in V_{1}$ and $f\in V_{0}$ are given. then there exists a
$T>0$ depending on $\left\Vert u_{0}\right\Vert _{1}$, $\nu$ and $\left\Vert
f\right\Vert $, such that there exists a unique strong solution $u\in
L^{\infty}(0,T;V_{1})\cap L^{2}\left(  0,T;V_{2}\right)  $.
\end{theorem}

This result was obtained for a type of inequality similar to%
\begin{equation}
\left\Vert u\left(  .,t\right)  \right\Vert _{1}^{2}\leq\frac{1+\left\Vert
u_{0}\right\Vert _{1}^{2}}{\sqrt{1-Kt\left(  1+\left\Vert u_{0}\right\Vert
_{1}\right)  ^{2}}},\label{8}%
\end{equation}
where $K=(2\dfrac{\left\Vert f\right\Vert ^{2}}{\nu}+\dfrac{c_{1}}{\nu^{3}})$.
Hereafter, $c_{i}\in%
\mathbb{N}
$ ,will denote a dimensionless scale invariant positive constant which might
depend on the shape of the domain. The bound in $(\ref{8})$ is only finite
while
\begin{equation}
Kt\left(  1+\left\Vert u_{0}\right\Vert _{1}\right)  ^{2}<1;\label{9}%
\end{equation}
if we choose $T$ satisfying%
\begin{equation}
T<\frac{1}{K\left(  1+\left\Vert u_{0}\right\Vert _{1}^{2}\right)
}.\label{10}%
\end{equation}
The main result of this paper is given in the following theorem.

\begin{theorem}
Assume that $u_{0}\in V_{1}$ and $u$ is the corresponding strong solution to
$(\ref{7})$ on $[0,T]$, then\newline i) If $f\in V_{0}$ and
\begin{equation}
c_{8}T\Vert f\Vert^{2}+c_{9}\left\Vert u\left(  0\right)  \right\Vert
^{2}+\arctan\left\Vert u\left(  0\right)  \right\Vert _{1}^{2}<\frac{\pi}{2},
\label{11}%
\end{equation}
then $u$ exists for each finite time\ $T$ and remains smooth. \newline ii) If
$f\in L^{2}\left(  0,T,V_{0}\right)  $ \ and
\begin{equation}
c_{10}\int_{0}^{T}\Vert f\Vert^{2}ds+c_{11}\left\Vert u\left(  0\right)
\right\Vert ^{2}+\arctan\left\Vert u\left(  0\right)  \right\Vert _{1}%
^{2}+\arctan\left\Vert u\left(  0\right)  \right\Vert _{1}^{2}<\frac{\pi}{2},
\label{12}%
\end{equation}
then $u$ exists globally ($T$ can be $\infty$) and remains smooth.
\end{theorem}

\begin{proof}
Multiplying $(\ref{7})$ by $\triangle u$, we have
\begin{equation}
\frac{1}{2}\frac{d}{dt}\left\Vert u\left(  .,t\right)  \right\Vert _{1}%
^{2}+\nu\Vert\triangle u\Vert^{2}-\int_{\Omega}\left(  u.\nabla u\right)
.\triangle u\text{ }dx=\left(  f,\triangle u\right)  .\label{13}%
\end{equation}
Using schwartz and Young inequality we get%
\begin{align}
\left\vert \left(  f,\triangle u\right)  \right\vert  &  \leq\left\Vert
f\right\Vert _{L^{2}}\left\Vert \triangle u\right\Vert _{L^{2}}\label{14}\\
&  \leq c_{3}\left\Vert f\right\Vert _{L^{2}}^{2}+\frac{\nu}{2}\left\Vert
\triangle u\right\Vert _{L^{2}}^{2}.\nonumber
\end{align}
For the nonlinear term, we use the H\"{o}lder's inequality%
\begin{equation}%
\begin{array}
[c]{ll}%
\left\vert \int_{\Omega}\left(  u.\nabla u\right)  .\triangle udx\right\vert
& \leq c_{4}\Vert u\Vert_{L^{6}}\left\Vert \nabla u\right\Vert _{L^{3}%
}\left\Vert \triangle u\right\Vert _{L^{2}}\\
& \leq c_{5}\left\Vert \nabla u\right\Vert ^{\frac{3}{2}}\left\Vert \triangle
u\right\Vert ^{\frac{3}{2}}.
\end{array}
\label{15}%
\end{equation}
However, an application of Young's inequality to $(\ref{15})$ yields%
\begin{equation}
\left\vert \int_{\Omega}\left(  u.\nabla u\right)  .\triangle udx\right\vert
\leq c_{6}\left\Vert \nabla u\right\Vert ^{6}+\frac{\nu}{2}\left\Vert
\triangle u\right\Vert ^{2}.\label{16}%
\end{equation}
Combining $(\ref{13})$, $(\ref{14})$ and $(\ref{16})$, we have that
\begin{equation}
\frac{d}{dt}\left\Vert u\left(  .,t\right)  \right\Vert _{1}^{2}\leq
c_{3}\left\Vert f\right\Vert ^{2}+c_{6}\left\Vert u\right\Vert _{1}%
^{6}.\label{17}%
\end{equation}
Suppose first $f\in V_{0}$. Setting $y\left(  t\right)  =\left\Vert u\left(
.,t\right)  \right\Vert _{1}^{2}$ in $(\ref{17})$, this gives%
\begin{equation}
\frac{d}{dt}y\leq c_{3}\Vert f\Vert^{2}+c_{6}y^{3}.\label{18}%
\end{equation}
Dividing $(\ref{18})$ by $1+y^{2}$, we have%
\begin{equation}
\frac{\frac{d}{dt}y}{1+y^{2}}\leq\frac{c_{3}\Vert f\Vert^{2}}{1+y^{2}}%
+c_{6}y\frac{y^{2}}{1+y^{2}}.\label{19}%
\end{equation}
Since $\left\Vert \frac{1}{1+y^{2}}\right\Vert _{\infty}\leq1$ and $\left\Vert
\frac{y^{2}}{1+y^{2}}\right\Vert _{\infty}\leq1$, this yields%
\begin{equation}
\frac{y^{\prime}}{1+y^{2}}\leq c_{3}\Vert f\Vert^{2}+c_{6}y.\label{20}%
\end{equation}
Integrate this over $\left[  0,T\right]  $ to get
\begin{equation}
\arctan y\left(  t\right)  -\arctan y\left(  0\right)  \leq c_{3}T\Vert
f\Vert^{2}+c_{6}\int_{0}^{T}y\left(  s\right)  ds.\label{21}%
\end{equation}
The function $\arctan$ is defined for each $t\in%
\mathbb{R}
$ and $\frac{-\pi}{2}<$ $\arctan t<\frac{\pi}{2}$. Multiplying $(\ref{7})$ by
$u$  to get
\begin{equation}
\frac{1}{2}\frac{d}{dt}\left\Vert u\left(  .,t\right)  \right\Vert ^{2}%
+\nu\Vert u\Vert_{1}^{2}=\left(  f,u\right)  .\label{22}%
\end{equation}
Using Schwartz and Poincare inequality, we have%
\begin{align}
\frac{1}{2}\frac{d}{dt}\left\Vert u\left(  .,t\right)  \right\Vert ^{2}%
+\nu\Vert u\Vert_{1}^{2} &  \leq\left\Vert u\right\Vert \left\Vert
f\right\Vert \label{23}\\
&  \leq\frac{\nu}{2}\left\Vert u\right\Vert _{1}^{2}+c_{7}\left\Vert
f\right\Vert ^{2}.\nonumber
\end{align}
Integrate $(\ref{23})$ over $\left[  0,T\right]  $, to get
\begin{equation}
\left\Vert u\left(  .,t\right)  \right\Vert ^{2}+\frac{\nu}{2}\int_{0}%
^{T}\Vert u\left(  s\right)  \Vert_{1}^{2}ds\leq c_{7}\int_{0}^{T}\left\Vert
f\right\Vert ds+\left\Vert u\left(  0\right)  \right\Vert ^{2}.\label{24}%
\end{equation}
For $f\in V_{0}$, the inequality above gives
\begin{equation}
\arctan y\left(  t\right)  \leq c_{8}T\Vert f\Vert^{2}+c_{9}\left\Vert
u\left(  0\right)  \right\Vert ^{2}+\arctan y\left(  0\right)  .\label{25}%
\end{equation}
The function $\tan y$ is increasing and invertible for $\frac{-\pi}{2}<$
$y<\frac{\pi}{2}$ with inverse function $\arctan t$. Thus, for
\begin{equation}
c_{8}T\Vert f\Vert^{2}+c_{9}\left\Vert u\left(  0\right)  \right\Vert
^{2}+\arctan\left\Vert u\left(  0\right)  \right\Vert _{1}^{2}<\frac{\pi}%
{2}\label{26}%
\end{equation}
we apply the function $\tan$ on $(\ref{25})$ to get%
\begin{equation}
\left\Vert u\left(  t\right)  \right\Vert _{1}^{2}\leq\tan\left(  c_{8}T\Vert
f\Vert^{2}+c_{9}\left\Vert u\left(  0\right)  \right\Vert ^{2}+\arctan
\left\Vert u\left(  0\right)  \right\Vert _{1}^{2}\right)  .\label{27}%
\end{equation}
The assumption $(\ref{26})$ on the initial data guarantees that the right hand
side of $(\ref{27})$ is finite.\newline For $f\in L^{2}\left(  0,T,V_{0}%
\right)  $ and for each $T>0$, integrate $(\ref{20})$ to get%
\begin{equation}
\arctan y\left(  t\right)  -\arctan y\left(  0\right)  \leq c_{3}\int_{0}%
^{T}\Vert f\Vert^{2}ds+c_{6}\int_{0}^{T}y\left(  s\right)  ds.\label{28}%
\end{equation}
Therefore, $(\ref{24})$ implies%
\begin{equation}
\arctan y\left(  t\right)  \leq c_{10}\int_{0}^{T}\Vert f\Vert^{2}%
ds+c_{11}\left\Vert u\left(  0\right)  \right\Vert ^{2}+\arctan y\left(
0\right)  .\label{29}%
\end{equation}
In particular, for%
\begin{equation}
c_{10}\int_{0}^{T}\Vert f\Vert^{2}ds+c_{11}\left\Vert u\left(  0\right)
\right\Vert ^{2}+\arctan\left\Vert u\left(  0\right)  \right\Vert _{1}%
^{2}<\frac{\pi}{2},\label{30}%
\end{equation}
we apply the function $\tan$ on $(\ref{29})$ to get%
\begin{equation}
\left\Vert u\left(  t\right)  \right\Vert _{1}^{2}\leq\tan\left(  c_{10}%
\int_{0}^{T}\Vert f\Vert^{2}ds+c_{11}\left\Vert u\left(  0\right)  \right\Vert
^{2}+\arctan\left\Vert u\left(  0\right)  \right\Vert _{1}^{2}\right)
.\label{31}%
\end{equation}
It follows that the assumption $(\ref{30})$ guarantees that $\left\Vert
u\left(  t\right)  \right\Vert _{1}<\infty$ for all $t>0$.
\end{proof}

Recall that the\ classical regularity result for $f=0$ \cite[Theorem
3.12.]{13} was obtained for a type of inequality similar to%
\begin{equation}
\left\Vert u\left(  .,t\right)  \right\Vert _{1}^{4}\leq\frac{\left\Vert
u_{0}\right\Vert _{1}^{4}}{1-c_{12}2t\left\Vert u_{0}\right\Vert _{1}^{4}%
}.\label{32}%
\end{equation}
It follows that if $\left\Vert u_{0}\right\Vert _{1}$ is finite, then
$\left\Vert u\left(  .,t\right)  \right\Vert _{1}$ is finite, at least for
\begin{equation}
t<\nu^{3}/128\left\Vert u_{0}\right\Vert _{1}^{4}.\label{33}%
\end{equation}
Consequently, we get the following result for the negligible forces.

\begin{corollary}
Assume that $u_{0}\in V_{1}$ and $u$ is the corresponding strong solution to
$(\ref{7})$ on $[0,T]$, then $u$ exists globally and remains smooth\ for all
$T>0$ if
\begin{equation}
c_{11}\left\Vert u\left(  0\right)  \right\Vert ^{2}+\arctan\left\Vert
u\left(  0\right)  \right\Vert _{1}^{2}<\frac{\pi}{2}. \label{34}%
\end{equation}

\end{corollary}

\begin{proof}
From $(\ref{17})$ we get for $f=0$
\begin{equation}
\frac{d}{dt}\left\Vert u\right\Vert _{1}^{2}\leq c_{6}\left\Vert u\right\Vert
_{1}^{6}. \label{35}%
\end{equation}
In particular, we have%
\begin{equation}
\frac{d}{dt}y\leq c_{6}y\left(  1+y^{2}\right)  \text{ with }y\left(
t\right)  =\left\Vert u\left(  .,t\right)  \right\Vert _{1}^{2}. \label{36}%
\end{equation}
Dividing $(\ref{36})$ by $1+y^{2}$ yields%
\begin{equation}
\frac{\frac{d}{dt}y}{1+y^{2}}\leq c_{6}y. \label{37}%
\end{equation}
Integrate this over $\left[  0,T\right]  $ to get%
\begin{equation}
\arctan y\left(  t\right)  \leq c_{6}\int_{0}^{T}y\left(  s\right)  ds+\arctan
y\left(  0\right)  . \label{38}%
\end{equation}
From $(\ref{24})$, we find%
\begin{equation}
\int_{0}^{T}y\left(  s\right)  ds\leq c_{9}\left\Vert u\left(  0\right)
\right\Vert ^{2} \label{39}%
\end{equation}
this implies that $(\ref{38})$ is equivalent to%
\begin{equation}
\arctan y\left(  t\right)  \leq c_{11}\left\Vert u\left(  0\right)
\right\Vert ^{2}+\arctan y\left(  0\right)  . \label{40}%
\end{equation}
Now, applying the function $\tan$ on $(\ref{40})$ to get%
\begin{equation}
\left\Vert u\left(  t\right)  \right\Vert _{1}^{2}\leq\tan\left(
c_{11}\left\Vert u\left(  0\right)  \right\Vert ^{2}+\arctan\left\Vert
u\left(  0\right)  \right\Vert _{1}^{2}\right)  , \label{41}%
\end{equation}
which is finite thanks to the following assumption%
\begin{equation}
c_{11}\left\Vert u\left(  0\right)  \right\Vert ^{2}+\arctan\left\Vert
u\left(  0\right)  \right\Vert _{1}^{2}<\frac{\pi}{2} \label{42}%
\end{equation}
and this concludes the proof.
\end{proof}

Since the condition $(\ref{42})$ is independent of time we get global
estimates for $\left\Vert u\left(  .,t\right)  \right\Vert _{1}$ by this
method. An important consequence of this result is that for each finite time
$T^{\ast}$ such that
\begin{equation}
T^{\ast}<\nu^{3}/128\left\Vert u_{0}\right\Vert _{1}^{4}, \label{43}%
\end{equation}
there is a $u_{0}$ satisfies
\begin{equation}
c_{11}\left\Vert u\left(  0\right)  \right\Vert ^{2}+\arctan\sqrt{\nu
^{3}/128T^{\ast}}<\frac{\pi}{2}. \label{44}%
\end{equation}
Thus the solution associate to $u_{0}$ satisfies $(\ref{43})$ has a\ global
regularity. But for the same value of $\left\Vert u\left(  0\right)
\right\Vert _{1}$ occurs a blow up in finite time $T^{\ast}$ by the usual
method $(\ref{32})$. This property follows easily when $\left\Vert u\left(
0\right)  \right\Vert $ approaches zero.

This result gives a simple condition for global regularity and extends the
known corresponding result $(\ref{32})$, where a blow-up criterion in finite
time $T$ depend on $u_{0}$ for negligible forces, see \cite{3, 5, 11, 13}.

\end{document}